\documentclass[anon]{l4dc2026_ArXiv}
\usepackage{cite}
\usepackage{amsmath,amssymb,amsfonts}
\usepackage{aliascnt}

\usepackage{mathtools}
\usepackage{graphicx}
\usepackage{mathbbol}
\usepackage{stfloats}
\usepackage{verbatim}
\usepackage{textcomp}
\usepackage{xcolor}
\usepackage{comment}
\usepackage{multirow}
\usepackage{makecell}

\usepackage{booktabs}
\usepackage{balance}
\usepackage{mathbbol}
\usepackage{algorithm}
\usepackage{algorithmicx}
\usepackage{algpseudocode}
\usepackage{mathrsfs}
\usepackage{threeparttable}
\usepackage{wrapfig}

\usepackage{enumitem}

\usepackage[font=small,labelfont=bf]{caption}

\newcommand{\method}[1]{\texttt{#1}}
\newcommand{\tr}{{{\mathsf T}}}

\newcommand{\col}{\textnormal{col}}
\newcommand{\f}{\textnormal{f}}

\newcommand{\C}{\textnormal{cl}}

\newcommand{\ub}{\mathbf{u}}
\newcommand{\ubb}{\bar{\mathbf{u}}}

\newcommand{\hP}{{\hat{P}}}
\newcommand{\sk}{{\small \texttt{S-KMPC}}}

\newtheorem{thm}{Theorem}%[section]
\newtheorem{assum}{Assumption}
\newtheorem{prop}{Proposition}
\newtheorem{lem}{Lemma}

\newtheorem{deff}{Definition}
\newtheorem{rem}{Remark}
\newtheorem{exam}{Example}

\usepackage[nameinlink]{cleveref}
\crefname{equation}{}{}

\crefname{thm}{Theorem}{Theorems}
\crefname{prop}{Proposition}{Propositions}
\crefname{lem}{Lemma}{Lemmas}
\crefname{deff}{Definition}{Definitions}
\crefname{rem}{Remark}{Remarks}
\crefname{cor}{Corollary}{Corollaryies}
\crefname{assum}{Assumption}{Assumptions}
\Crefname{equation}{}{}
\Crefname{thm}{Theorem}{Theorems}
\Crefname{cor}{Corollary}{Corollaries}
\Crefname{exam}{Example}{Examples}
\Crefname{lem}{Lemma}{Lemmas}
\Crefname{prop}{Proposition}{Propositions}
\Crefname{figure}{Figure}{Figures}
\Crefname{fact}{Fact}{Facts}
\Crefname{table}{Table}{Tables}
\Crefname{section}{Section}{Sections}
\Crefname{appendix}{Appendix}{Appendices}
\Crefname{assum}{Assumption}{Assumptions}

\makeatletter
\g@addto@macro\normalsize{%
  \setlength\abovedisplayskip{5pt}%
  \setlength\belowdisplayskip{5pt}%
  \setlength\abovedisplayshortskip{5pt}%
  \setlength\belowdisplayshortskip{5pt}%
}
\makeatother

\AtBeginEnvironment{thm}{\vspace{-2mm}}
\AtEndEnvironment{thm}{\vspace{-2mm}}
\AtBeginEnvironment{assum}{\vspace{-2mm}}
\AtEndEnvironment{assum}{\vspace{-2mm}}
\AtBeginEnvironment{deff}{\vspace{-2mm}}
\AtEndEnvironment{deff}{\vspace{-2mm}}
\AtBeginEnvironment{prop}{\vspace{-2mm}}
\AtEndEnvironment{prop}{\vspace{-2mm}}
\AtBeginEnvironment{lem}{\vspace{-2mm}}
\AtEndEnvironment{lem}{\vspace{-2mm}}
\newcommand{\longthmtitle}[1]{\mbox{}\emph{(#1):}}

\title[Stability of Koopman MPC]{On the Exponential Stability of Koopman Model Predictive Control}
\usepackage{times}
\author{%
 \Name{Xu Shang} \Email{x3shang@ucsd.edu}\\
 \addr Department of Electrical and Computer Engineering, University of California San Diego
 \AND
 \Name{Jorge Cort\'es} \Email{cortes@ucsd.edu}\\
 \addr Department of Mechanical and Aerospace Engineering, University of California San Diego
 \AND
 \Name{Yang Zheng} \Email{zhengy@ucsd.edu}\\
 \addr Department of Electrical and Computer Engineering, University of California San Diego
}

\begin{document}

\maketitle
\vspace{-2mm}
\begin{abstract}
Koopman Model Predictive Control (MPC) uses a lifted linear predictor to efficiently handle constrained nonlinear systems. While constraint satisfaction and (practical) asymptotic stability have been studied, explicit guarantees of local exponential stability seem to be missing. This paper revisits the exponential stability for Koopman MPC. We first analyze a Koopman LQR problem and show that 1) with zero modeling error, the lifted LQR policy is globally optimal and globally asymptotically stabilizes the nonlinear plant, and 2) with the lifting function and one-step prediction error both Lipschitz at the origin, the closed-loop system is locally exponentially stable. These results facilitate terminal cost/set design in the lifted Koopman space. Leveraging linear-MPC properties (boundedness, value decrease, recursive feasibility), we then prove local exponential stability for a stabilizing Koopman MPC under the same conditions as Koopman LQR. Experiments on an inverted pendulum show better convergence performance and lower accumulated cost than the traditional Taylor-linearized MPC approaches. 
\end{abstract}


\vspace{-1.5mm}
\section{Introduction}
Model Predictive Control (MPC) is a well-established feedback strategy where, at each sampling instant, a finite-horizon optimal control problem is solved using the current state, and only the first control input is applied before repeating the process \citep{rawlings2017model}. Its explicit handling of input and state constraints has led to success across a wide range of applications \citep{mayne2000constrained, qin2003survey, zheng2016distributed}. For nonlinear systems, however, the MPC problem becomes nonlinear and nonconvex, making exact solutions difficult. In practice, \emph{suboptimal yet feasible} implementations are often adopted, which can still guarantee closed-loop performance with careful design \citep{scokaert2002suboptimal}. This \emph{feasibility-implies-stability} principle underpins many real-time methods, such as successive linearization \citep{diehl2005real}.

To overcome the computational challenges of nonlinear MPC, recent work has explored alternative linearization strategies beyond first-order Taylor expansions. One promising approach is \textit{Koopman linearization}, which lifts the nonlinear dynamics into a higher-dimensional space where they evolve approximately linearly. Originally introduced for autonomous systems \citep{koopman1931hamiltonian}, the Koopman operator framework has been extended to controlled systems \citep{korda2018linear, williams2015data,haseli2023modeling}. A key advantage is that Koopman linear models can be efficiently estimated using data-driven techniques, such as Extended Dynamic Mode Decomposition (EDMD) \citep{williams2015data}. This enables the design of (data-driven) Koopman~MPC, where each step solves a convex program using a linear predictor identified from~offline~data \citep{korda2018linear}. By combining the expressiveness of nonlinear modeling with computational benefits of convex optimization, Koopman MPC has attracted a growing interest for its practical scalability and performance \citep{haggerty2023control, mamakoukas2019local,shang2025dictionary}.

It is known that general nonlinear control dynamics cannot be exactly represented by finite-dimensional Koopman linear or bilinear models~\citep{haseli2023modeling}. Recent work has begun to investigate the closed-loop performance of Koopman MPC for nonlinear systems under modeling errors \citep{zhang2022robust,mamakoukas2022robust,worthmann2024data, bold2024data, de2024koopman}. For instance, \citet{zhang2022robust} proposed a robust tube-based Koopman MPC strategy, which relies on tightened constraints in the Koopman space to ensure \textit{constraint satisfaction} in the presence of small modeling errors;  similarly, \citet{mamakoukas2022robust} enforced {constraint satisfaction} via conservative surrogate constraints using a Hankel-Koopman model. 
 More recently, \citet{de2024koopman} established \emph{input-to-state stability} for a Koopman MPC variant with an interpolated initial condition, while \citet{worthmann2024data} proved \emph{practical asymptotic stability} using terminal ingredients designed from the original nonlinear system. A terminal-free variant was analyzed in \citet{bold2024data,schimperna2025data}, where \emph{(practical) asymptotic stability} is obtained under a cost-controllability assumption on the original nonlinear dynamics.

As discussed above, several notions of closed-loop stability for Koopman MPC have been established under various assumptions. Despite the progress, to our best knowledge, the basic question of \emph{local exponential stability} seems to have been overlooked. In particular, local exponential stability cannot be implied or directly derived by the results in \citep{zhang2022robust,mamakoukas2022robust,worthmann2024data, bold2024data, de2024koopman}. 
In this paper, we revisit this fundamental question: we design practical terminal ingredients in the lifted Koopman space, and provide clear conditions under which Koopman MPC ensures \emph{local exponential stability} of the nonlinear closed-loop system.  For brevity, we refer to this scheme as stabilizing Koopman MPC~(\sk{}).

Our technical results are as follows. We begin with a Koopman linear quadratic regulator~(LQR), closely linked to \sk{}. This is an \emph{unconstrained, infinite-horizon} optimal control problem. With no Koopman modeling error, a globally optimal policy can be obtained by solving a standard LQR in the lifted space (\Cref{lemma:optimal-Koopman-control}); under mild assumptions, this policy \emph{globally asymptotically} stabilizes the original nonlinear dynamics (\Cref{proposition:global-stable}).  If both the lifting function and the one-step Koopman prediction error are Lipschitz around the origin, the nonlinear closed loop is \emph{locally exponentially stable} (\Cref{them:exp-LQR}). These results guide the terminal design in \sk{}: we construct the terminal cost and terminal set \emph{in the lifted space} via the Koopman-LQR. Since \sk{} is designed based on a lifted linear model, it naturally inherits several key properties of linear MPC \citep{rawlings2017model}, including boundedness, decrease of the value function, and recursive feasibility of the Koopman update (\Cref{prop:terminal,prop:boundedness,prop:rec-feasibility}). Together with \emph{continuous} lifting, these yield \emph{local asymptotic stability} of the nonlinear closed loop (\Cref{prop:output-stabi}). Strengthening the assumptions to the Lipschitz lifting function and one-step Koopman prediction error around the origin, \sk{} guarantees the \emph{local exponential stability} of the nonlinear closed-loop system (\Cref{them:exp-stable}). We provide a detailed comparison with existing Koopman stability results later in \Cref{remark:comparison}. 

The rest of this paper is structured as follows. \Cref{sec:preliminary} provides preliminaries on the Koopman MPC. The Koopman LQR problem is discussed in \Cref{sec:LQR-stab} and the  \sk{} design is provided in \Cref{section:analysis-original-state}. Numerical results are shown in \Cref{sec:Num-exp} and we gather our conclusions in \Cref{sec:conclu}.

\vspace{-4mm}
\section{Preliminaries and Problem Statement}
\label{sec:preliminary}
\vspace{-2mm}
We consider a discrete-time nonlinear system of the form
\begin{equation}
\label{eqn:nonlinear}
    x_{t+1} = f(x_t, u_t),  
\end{equation}
where $x_t \in 
\mathbb{R}^n$ is the system sate, $u_t \in \mathbb{R}^m$ is the system input at time $t$, and $f: \mathbb{R}^n \times \mathbb{R}^m \rightarrow \mathbb{R}^n$ denotes the system dynamics. We make the following standard assumption.
\begin{assum}
    \label{assum:nonlinear-prop}
    The function $f$ is continuously differentiable and satisfies $f(\mathbb{0}, \mathbb{0}) = \mathbb{0}$. 
\end{assum}
\vspace{-3mm}

\subsection{Exponential stability}
\vspace{-1mm}
A basic objective is to design a feedback policy $u_t=\pi(x_t)$ that exponentially stabilizes \cref{eqn:nonlinear}. 
\begin{deff}[Local exponential stability] 
\label{def:exp-stab}
The origin~of system \cref{eqn:nonlinear} with a control law $u_t=\pi(x_t)$ is \emph{locally exponentially stable} if there exist $c > 0$, $\rho\in(0,1)$, and a neighborhood $\mathcal{N}$ of the origin, such that the closed-loop dynamics satisfies
$
  \|x_t\| \le c\,\rho^{\,t}\,\|x_0\|$, for all $x_0\in\mathcal{N},  t\in \mathbb{Z}_{\ge 0}$. 
\end{deff}

\vspace{-1mm}
\begin{lem}{\citep[Theorem B.19]{rawlings2017model}} 
\label{lemma:LES}
    Consider the autonomous system $x_{t+1}=g(x_t)$ with $g(\mathbb{0})=\mathbb{0}$. 
Suppose there exist constants $\alpha_1,\alpha_2,\alpha_3>0$, an invariant set $\mathcal D\subseteq\mathbb R^n$ with $\mathbb{0} \in \mathrm{int}(D)$, and a function $V: \mathcal{D} \to \mathbb{R}_{+}$ satisfying
\begin{subequations}
  \begin{align}
  \alpha_1\|x\|^2 \le V(x)\le \alpha_2\|x\|^2, \quad \forall x \in \mathcal{D}, \label{eq:quad-bounds}
\\
  V\big(g(x)\big)-V(x)\le -\,\alpha_3\|x\|^2, \quad \forall x \in \mathcal{D}. \label{eq:decrease}
\end{align}  
\end{subequations}
Then, the system is locally exponentially stable, and $\mathcal{D}$ is a region of attraction (ROA). 
\end{lem}
\vspace{-2mm}

This standard result is used throughout the paper; for completeness, we provide a brief proof in \Cref{Appendix:LES}.  The function $V$ in \Cref{lemma:LES} is called a Lyapunov function. Both the quadratic upper and lower bounds in \cref{eq:quad-bounds} are important, and the decrease condition~\cref{eq:decrease} ensures stability. Notably, there exist
$
c:\!=\!\sqrt{\alpha_2/\alpha_1}\ge1,  
\rho:\!=\!\sqrt{1-\alpha_3/\alpha_2 }\!\in\!(0,1) 
$ 
and $r>0$ such that every trajectory with $\|x_0\|<r$ is well-defined for all $k\ge0$ and satisfies $\|x_t\|\le c\,\rho^{\,t}\,\|x_0\|, t \in \mathbb{Z}_{\ge 0}$.

\vspace{-2mm}
\subsection{Nonlinear MPC basics and Koopman linearization}
\label{subsec:MPC-basics}
\vspace{-1mm}

\begin{wrapfigure}[7]{r}{.62\textwidth}
\vspace{-5mm}
\begin{subequations} \label{eq:nonlinear-MPC}
    \begin{align}
    \tilde{V}_N^*(x) =    \min_{\mathbf{u}} &\quad \sum_{k = 0}^{N-1} l(x_{t+k},u_{t+k}) + \tilde{V}_\mathrm{f}(x_{t+N}) \\
        \text{subject to} & \quad x_{t+k+1} = f(x_{t+k}, u_{t+k}), \label{eq:nonlinear-MPC-a}\\
        & \quad u_{t+k} \in \mathcal{U}, \, x_{t+k} \in \mathcal{X}, \, k \in \mathbb{Z}_{[0,N-1]},
         \\
        & \quad  x_t = x, \, x_{t+N} \in \tilde{\mathcal{X}}_\mathrm{f}.
    \end{align}
\end{subequations}
\end{wrapfigure}
At each time~$t$, the MPC algorithm solves the nonlinear problem \eqref{eq:nonlinear-MPC} with initial state $x_t = x \in \mathbb{R}^n$ and horizon $N  \in \mathbb{N}$. In  \eqref{eq:nonlinear-MPC},  $\mathcal{X}$ and $\mathcal{U}$ are the state and input constraints, and $\tilde{V}_\mathrm{f}$ and $\tilde{\mathcal{X}}_\mathrm{f}$ denote suitable terminal cost and set.  We assume both $\mathcal{X}$ and $\mathcal{U}$ are convex, and $\mathcal{U}$ is also bounded. We choose $\ell: \mathbb{R}^n \times \mathbb{R}^m \to \mathbb{R}_+$ as a quadratic stage~cost: 
\begin{equation} \label{eq:stage-cost}
\ell(x,u) : = \|x\|^2_Q + \|u\|^2_R, 
\end{equation}
with $Q$ and $R$ being two positive definite matrices. If \cref{eq:nonlinear-MPC} is feasible at $x$, then the optimal value function $\tilde V_N^*(x)$ is  finite and we say $x\in\mathrm{dom}(\tilde V_N^*)$. In this case, let $\tilde{\mathbf u}^*_x$ be an optimal input sequence to \cref{eq:nonlinear-MPC}. The MPC feedback is the first control action $\tilde \kappa(x)\;=\;\tilde{\mathbf u}^*_x(0), \ x\in\mathrm{dom}(\tilde V_N^*)$.

With appropriately chosen terminal ingredients $\tilde V_{\mathrm f}$ and $\tilde{\mathcal X}_{\mathrm f}$, the closed-loop dynamics with the MPC law, \emph{i.e.},
$x_{t+1}=f(x_t,\tilde\kappa(x_t))$, is \textit{locally exponentially stable}, and the value function $\tilde V_N^*$ serves as a Lyapunov function; see \citet[Ch.~2]{rawlings2017model} for details. A well-known challenge is~that \cref{eq:nonlinear-MPC} is generally \emph{nonconvex} and thus hard to solve for global optimality. We next introduce~a popular \textit{Koopman linearization} strategy which is increasingly used in applications;~see~\emph{e.g.},~\citet{shi2024koopman}. 
 
The key idea is to approximate the state sequence in~\cref{eq:nonlinear-MPC-a} using a high-dimensional linear predictor in the Koopman framework \citep{koopman1931hamiltonian}. Define the lifted~state
\begin{equation} \label{eq:dictionary}
    z_t = \Psi(x_t) := \col(\psi_1(x_t), \ldots, \psi_{n_z}(x_t)),
\end{equation}
where $n_z\!\geq\!n$ and each $\psi_i:\mathbb{R}^n \! \to \! \mathbb{R}$ is a chosen \textit{observable}. The dictionary $\Psi$ typically includes the identity mapping so that $x_t$ can be reconstructed from $z_t$ (\emph{i.e.}, $x_t=C z_t$) \citep{korda2018linear,strasser2024koopman,mamakoukas2022robust}. We can represent the nonlinear dynamics \cref{eqn:nonlinear} as
\begin{equation} \label{eq:nonlinear-system-Koopman}
    \begin{aligned}
        z_{t+1} = Az_t + Bu_t + e(Cz_t,u_t), \qquad %\\
        x_t = Cz_t,
    \end{aligned}
\end{equation}
where $A, B, C$ are matrices with compatible size, and the one-step modeling error is defined~as 
\begin{equation} \label{eq:process-error-term}
    e(Cz_t,u_t)\!=\! e(x_t,u_t)\!:=\! \Psi(f(x_t, u_t))\! -\! A\Psi(x_t)\! -\! Bu_t.  % , \qquad w(z_t) = **
\end{equation}
The modeling error depends on the choice of observables~\cref{eq:dictionary} and the matrices $A, B$. If $e\equiv 0$, we say  \cref{eqn:nonlinear} admits an \textit{exact Koopman linear embedding} \citep{shang2024willems}, leading to an exact linear predictor in the lifted Koopman space. 

\begin{wrapfigure}[6]{r}{.62\textwidth}
\vspace{-9.5mm}
\begin{subequations} \label{eq:Koopman-MPC}
    \begin{align}
    {V}_N^*(z) =    \min_{\mathbf{u}} &\quad \sum_{k = 0}^{N-1} l(Cz_{t+k},u_{t+k}) + {V}_\mathrm{f}(z_{t+N}) \label{eq:Koopman-MPC-cost} \\
        \text{subject to} & \quad z_{t+k+1} = Az_{t+k}+Bu_{t+k}, \label{eq:Koopman-MPC-cons-1}\\
        & \quad  u_{t+k} \!\in\! \mathcal{U}, \; Cz_{t+k} \!\in\! \mathcal{X}, \, k \!\in\! \mathbb{Z}_{[0,N-1]},\label{eq:Koopman-MPC-cons-2}\\
        & \quad z_t = z, \, z_{t+N} \in {\mathcal{Z}}_\mathrm{f}. \label{eq:Koopman-MPC-cons-3}
    \end{align}
\end{subequations}
\end{wrapfigure}
In practice, if $e$ is ``small'' locally, we can handle it within the MPC design. 
At each time $t$ with an initial state $x_t = x$, we lift this initial state to the Koopman space as $z = \Psi(x)$, and replace the nonlinear optimization \cref{eq:nonlinear-MPC} with problem \eqref{eq:Koopman-MPC}, where ${V}_\mathrm{f}$ and ${\mathcal{Z}}_\mathrm{f} \subseteq \mathbb{R}^{n_z}$ are appropriate terminal cost and terminal set to be designed. Let $\mathbf{u}^*_z$ be one optimal solution for feasible $z$. The Koopman MPC~law~is 
\begin{equation} \label{eq:Koopman-MPC-law}
    \kappa(z) = \mathbf{u}^*_z(0), \quad \forall z \in  \mathrm{dom}({V}_N^*),
\end{equation}
and the closed-loop dynamics of the original nonlinear system \cref{eqn:nonlinear} with the control law \eqref{eq:Koopman-MPC-law} becomes
\begin{gather} 
    x_{t+1} = f(x_t, \kappa(\Psi(x_t))) \label{eq:closed-loop-Koopman}. 
\end{gather}

The basic Koopman MPC formulation was first proposed in \citet{korda2018linear}, without discussing the terminal ingredients and the closed-loop stability. 

\vspace{-4mm}
\subsection{Problem statement}
\vspace{-1mm}
The Koopman MPC problem \cref{eq:Koopman-MPC} is computationally efficient because the predictor \cref{eq:Koopman-MPC-cons-1} is \emph{linear} in $\mathbf{u}$. The lifted linear model matrices $A,B,C$ in \cref{eq:nonlinear-system-Koopman} can be identified from data via EDMD and related methods \citep{williams2015data}. Some recent works \citep{de2024koopman,zhang2022robust,mamakoukas2022robust,worthmann2024data} have addressed constraint satisfaction and (practical) asymptotic stability of the closed-loop system \cref{eq:closed-loop-Koopman}.

In this paper, we revisit the \textit{exponential stability} of \cref{eq:closed-loop-Koopman} with the Koopman MPC law, and discuss the design of the terminal cost ${V}_\mathrm{f}$ and terminal set ${\mathcal{Z}}_\mathrm{f}$ in \cref{eq:Koopman-MPC}. Unlike \citet{de2024koopman,zhang2022robust,mamakoukas2022robust,worthmann2024data}, we assume the lifting observable \cref{eq:dictionary} is given and a linear model $A, B, C$ in \cref{eq:nonlinear-system-Koopman} has been estimated from data. This isolates the core stability questions of the Koopman MPC from the complexity arising in the identification procedure. Our idea follows the standard stabilizing MPC for linear systems \citep[Ch.~2]{rawlings2017model}, adapted carefully to the Koopman-lifted space while accounting for the error in \cref{eq:nonlinear-system-Koopman}. Besides the terminal ingredients ${V}_\mathrm{f}$ and ${\mathcal{Z}}_\mathrm{f}$, the lifting function \cref{eq:dictionary} and the modeling error  \cref{eq:nonlinear-system-Koopman} also affects the closed-loop stability in \cref{eq:closed-loop-Koopman}. We will compare with the existing Koopman stability results in \Cref{remark:comparison}.  

Throughout this paper, for a state $x$ and policy $u=\pi(x)$, we denote the successor of system~\cref{eqn:nonlinear} by $x^+$, with lifted state $z^+ := \Psi(x^+)$. The one-step Koopman prediction is denoted by $\bar z^+ \!:=\! A\,\Psi(x) + B\,\pi(x)$. Note that $z^+ \!\neq\! \bar{z}^+$ unless $e\equiv 0$ in \cref{eq:process-error-term}. We make another assumption, similar to \citep{de2024koopman,zhang2022robust,mamakoukas2022robust}. 
\begin{assum}
\label{assum:Koop-obser}
     The lifting function $\Psi: \mathbb{R}^n  \rightarrow  \mathbb{R}^{n_z}$ is continuous, with $\Psi(\mathbb{0}) = \mathbb{0}$, and $x = C \Psi(x)$ where $C \in \mathbb{R}^{n_z \times n}$. In \cref{eq:nonlinear-system-Koopman}, the pair $(A, B)$  is stabilizable and the pair $(A, C)$ is observable.  
\end{assum}

\vspace{-2mm}
\section{Infinite-horizon Koopman LQR}
\label{sec:LQR-stab}
\vspace{-1mm}
This section discusses an infinite-horizon Koopman LQR problem. The results will facilitate the design of the terminal set and terminal cost for stabilizing Koopman MPC in the next section. 
\vspace{-3mm}
\subsection{Koopman optimal control}
\vspace{-1mm}
\begin{wrapfigure}[9]{r}{.48\textwidth}
\vspace{-14mm}
\begin{equation} \label{eq:optimal-control}
    \begin{aligned}
        \min_{\mathbf{u}_\infty}\quad& \sum_{k=0}^\infty l(x_{t+k}, u_{t+k}) \\
        \text{subject to}\quad &x_{t+k+1} \!=\! f(x_{t+k}, u_{t+k}), \\
        &x_t\! =\! x, k\! \in \! \mathbb{Z}_{\ge 0}.
    \end{aligned}
\end{equation}
\vspace{-3mm}
\begin{equation}
\label{eqn:LQR-KL}
\begin{aligned}
\min_{\mathbf{u}_\infty} & \quad \sum_{k = 0}^{\infty} l(Cz_{t+k}, u_{t+k}) \\
\text{subject to}& \quad z_{t+k+1} \!=\! Az_{t+k}\! +\! Bu_{t+k}, \\
& \quad z_t \!=\! z, k \! \in \!\mathbb{Z}_{\ge 0}.
\end{aligned}
\end{equation}
\end{wrapfigure}
Consider the infinite-horizon optimal control problem \eqref{eq:optimal-control}. We denote $\mathbf{u}_\infty := (u_t, u_{t+1}, \ldots)$ as the control sequence, $x_t = x$ as the initial state at time $t$. The stage cost $l$ is defined in~\cref{eq:stage-cost}. 
Compared with the MPC \cref{eq:nonlinear-MPC}, this formulation \cref{eq:optimal-control} has an infinite horizon and no state/input constraints. 

This nonlinear optimal control problem \cref{eq:optimal-control} is generally hard to solve. We can utilize a Koopman linear model to approximate the nonlinear dynamics in \eqref{eq:optimal-control}, leading to the Koopman LQR problem \eqref{eqn:LQR-KL}, where $z = \Psi(x)$. When the Koopman linear model is exact, we can obtain an \textit{explicit} optimal policy for~\cref{eq:optimal-control}.
\begin{lem} \label{lemma:optimal-Koopman-control}
     Suppose there exists a Koopman lifting \cref{eq:dictionary} such that $ e(Cz_t,u_t) \equiv 0$ in \cref{eq:process-error-term}, and \Cref{assum:nonlinear-prop,assum:Koop-obser} hold. Then, problem \cref{eq:optimal-control} has a \textit{globally} optimal feedback policy 
     \begin{equation} \label{eq:optimal-Koopman-control}
     u_{t+k} = K z_{t+k} = K\Psi(x_{t+k}),  \quad k \in \mathbb{Z}_{\ge 0}
     \end{equation}
     where $K$ is the optimal LQR feedback gain for \eqref{eqn:LQR-KL} associated with $A, B, C, Q, R$, \emph{i.e.}, 
     $K = -(R+B^\tr P B)^{-1} B^\tr PA$, with $P$ be the unique {positive definite} solution to the Riccati equation 
     \begin{equation}
     \label{eqn:Riccati}
     P = C^\tr Q C+ A^\tr P A - A^\tr PB(R + B^\tr P B)^{-1}B^\tr P A.
     \end{equation}
\end{lem}
\vspace{-7mm}

The proof is not difficult as \eqref{eqn:LQR-KL} is a standard LQR problem, and we provide some details in \Cref{Appendix:optimal-Koopman}. Note that the optimal policy \cref{eq:optimal-Koopman-control} is linear in the lifted Koopman space, but~remains nonlinear in the original state space. 
 The associated optimal value function for \cref{eqn:LQR-KL} is $V_{\infty}^*(z)\! =\! \|z\|_P^2$. While \Cref{lemma:optimal-Koopman-control} is not difficult to establish, it actually gives a globally optimal (nonlinear) policy to a class of nonlinear control problems \cref{eq:optimal-control}. We can further show that the nonlinear feedback law \cref{eq:optimal-Koopman-control} \textit{globally asymptotically} stabilizes the nonlinear system \cref{eqn:nonlinear} if there is no modeling error.
\begin{thm} \label{proposition:global-stable}
    Under the same conditions of \Cref{lemma:optimal-Koopman-control}, consider the feedback law \cref{eq:optimal-Koopman-control}. The closed-loop system $x_{t+1} = f(x_t, K\Psi(x_t))$ of \cref{eqn:nonlinear} is \textit{globally  asymptotically~stable}.
\end{thm}

\Cref{proposition:global-stable} guarantees only \emph{global asymptotic} stability of the nonlinear closed loop. The proof needs to establish both globally attractive and locally Lyapunov stable properties. Establishing the global attractiveness is easy, but the stability in the sense of Lyapunov requires additional arguments. Due~to page limit, we present the proof details in \Cref{Appendix:global-stable}. 

Although the lifted linear system \(z_{t+1}=(A+BK)z_t\) is globally exponentially stable, this does \emph{not} in general imply global exponential stability of the physical state \(x_t\).~The reason is that the exponential decay of \(z_t\) does not automatically transfer to \(x_t=Cz_t\) unless the lifting \(\Psi\) satisfies additional regularity properties. For example, \(x_t=Cz_t=0\) does not imply \(z_t=0\); some components of \(z_t\) in \(\ker C\) may be nonzero yielding \(x_{t+1}\neq 0\) with the propagation of the dynamics.

\vspace{-3mm}
\subsection{Local exponential stability of Koopman LQR}
\vspace{-1mm}
Under additional conditions, the closed-loop system with the controller \cref{eq:optimal-Koopman-control} is \emph{locally exponentially stable} even with Koopman modeling error. In the following, we denote $\mathcal{B}_r = \{x \in \mathbb{R}^n \mid \|x\|\leq r\}$. 
\begin{assum}
\label{assum:lift-cd}
    There exists constants $r_\psi$ and $L_\psi$ such that the lifting function $\Psi:\mathbb{R}^n \rightarrow \mathbb{R}^{n_z}$ satisfies $\|\Psi(x)\| \le L_\psi \|x\|, \ \forall x \in \mathcal{B}_{r_\psi}$.
\end{assum}
\begin{assum}
\label{assum:pred-erro}
There exist constants $r >0$ and $L>0$ such that the closed-loop one-step prediction error 
        $\bar{e}_\C (x_t) 
:=\Psi(f(x_t, K\Psi(x_t)) - A\Psi(x_t) - BK\Psi(x_t)$ satisfies $
    \|\bar{e}_\C(x)\| \le L \|x\|, \ \forall x \in \mathcal{B}_r$.
\end{assum}

\Cref{assum:lift-cd} means that the lifting function is locally bounded with respect to the physical state  $x$, which holds for all locally Lipschitz continuous functions (note that $\Psi(\mathbb{0}) = \mathbb{0}$ in \Cref{assum:Koop-obser}). \Cref{assum:pred-erro} requires that the one-step closed-loop prediction is sufficiently accurate.
\begin{thm}
\label{them:exp-LQR}
    Suppose \Cref{assum:nonlinear-prop,assum:lift-cd,assum:Koop-obser,assum:pred-erro} hold. Consider the feedback law \cref{eq:optimal-Koopman-control}. There exists $\delta \!>\!0$, such that if $L$ in \Cref{assum:pred-erro} satisfies $L < \delta$, the closed-loop system $x_{t+1} = f(x_t, K\Psi(x_t))$ of \cref{eqn:nonlinear} is \textit{locally exponentially~stable}. 
\end{thm}

The key idea is to show that the optimal value function for \eqref{eqn:LQR-KL} \emph{i.e.}, $\tilde{V}(x) := \Psi(x)^\tr P \Psi(x)$, is a valid Lyapunov function for the nonlinear system \eqref{eqn:nonlinear}. In other words, we show that $\tilde{V}(x) $ satisfies  \Cref{lemma:LES}. To prove this, one key step is to decompose the one-step Lyapunov change $\tilde{V}(x^+) - \tilde{V}(x)$ as a \textit{Koopman decrease} induced by the Koopman linear model and a \textit{perturbation} resulting from the Koopman prediction error. The Riccati identity ensures a substantial Koopman decrease, while \Cref{assum:lift-cd,assum:pred-erro} together guarantee the perturbation can be made strictly smaller than the Koopman decrease. The proof details are given in \Cref{Appendix:exp-LQR}. 

\vspace{1mm}
We conclude this section with a simple example to show the performance of Koopman LQR for nonlinear systems.\\
\vspace{-11.5mm}
\begin{wrapfigure}[8]{r}{0.5\textwidth}
\vspace{0mm}
\centering
\hspace{-5mm}
\includegraphics[width=0.27\textwidth]{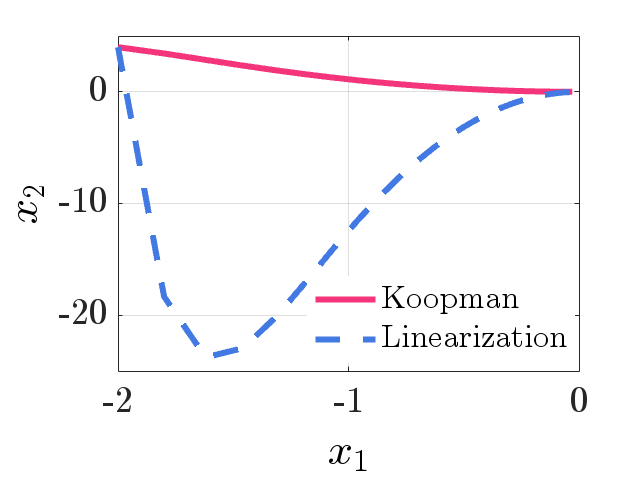} \label{fig:LQR-traj}
\hspace{-5mm}
\includegraphics[width=0.27\textwidth]{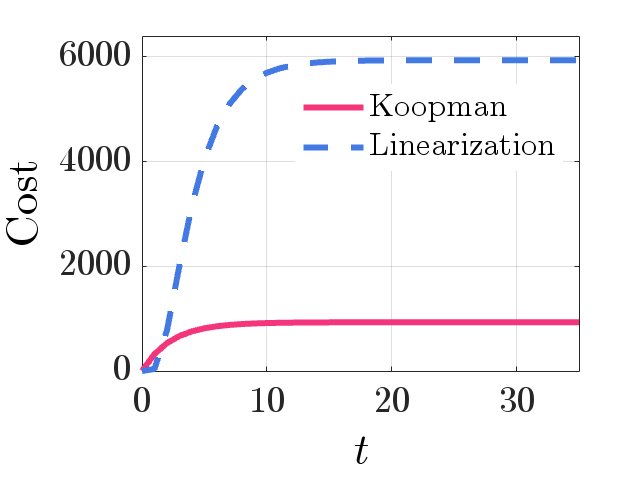}\label{fig:LQR-cost}
\hspace{-8mm}
\vspace{-4mm}
\captionsetup{width=0.9\linewidth}
\caption{LQR performance: exact Koopman model vs. first-order (Taylor) linearization. Left: closed-loop trajectory; right: accumulated cost. 
\vspace{-3mm}}
\label{fig:LQR}
\end{wrapfigure}
\begin{exam} \label{example:exact-linearization}
\hspace{-2mm}Consider the nonlinear system \\
\begin{minipage}{0.5\textwidth}
\[
        \begin{aligned}
            \begin{bmatrix}
                x_1 \\ x_2
            \end{bmatrix}^+ &\! =\! \begin{bmatrix}
                0.9  & 0 \\ 0 & 1.5
            \end{bmatrix}\begin{bmatrix}
                x_1 \\ x_2
            \end{bmatrix}\! +\! \begin{bmatrix}
                0 \\ 1
            \end{bmatrix} u \!+\! \begin{bmatrix}
                0 \\-5 x_1^2
            \end{bmatrix}.
        \end{aligned}
\]
\vspace{1mm}
\end{minipage}
We compare LQR controllers designed from (i) the exact lifted Koopman model and (ii) the local linearization, using $Q=\mathrm{diag}(1,1)$ and $R=1$. The \emph{exact} lifted-Koopman linear representation with the lifting function $\Psi(x)=\operatorname{col}(x_1,x_2,x_1^2)$ and the first-order (Taylor) linearization at the origin are
\[
\begingroup
    \setlength\arraycolsep{3pt}
\def\arraystretch{0.85} 
\label{ex:accu-Koop}
\textnormal{\text{Koopman:}}
    \begin{bmatrix}
        x_1 \\ x_2 \\ x_1^2
    \end{bmatrix}^+ \! =\! \begin{bmatrix}
        0.9  & 0 & 0 \\ 0 & 1.5 & -5 \\ 0 & 0 & 0.81
    \end{bmatrix}\begin{bmatrix}
        x_1 \\ x_2 \\ x_1^2
    \end{bmatrix} \!+\! \begin{bmatrix}
        0 \\ 1 \\ 0
    \end{bmatrix} u, \ \ \textnormal{\text{Taylor:}} \begin{bmatrix}
        x_1 \\ x_2 
    \end{bmatrix}^+ \! = \! \begin{bmatrix}
        0.9  & 0 \\ 0 & 1.5 
    \end{bmatrix}\begin{bmatrix}
        x_1 \\ x_2 
    \end{bmatrix} + \begin{bmatrix}
        0 \\ 1 
\end{bmatrix} u.
\endgroup
\]
As illustrated in Fig.~\ref{fig:LQR}, the Koopman LQR yields faster convergence with much better transient behavior than the linearized LQR. This is because the lifted model captures the quadratic coupling exactly; by \Cref{lemma:optimal-Koopman-control}, the Koopman LQR gives the optimal solution of \cref{eq:optimal-control} in this instance. 
\hfill $\square$
\end{exam}

\vspace{-5mm}
\section{Stabilizing Koopman MPC}
\label{section:analysis-original-state}
\vspace{-1mm}
It is well known that \emph{terminal ingredients} are crucial for closed-loop stability in MPC.  In this section, we design the terminal ingredients for Koopman MPC using the Koopman LQR results of \Cref{sec:LQR-stab}.

\vspace{-1mm}
\subsection{Design of the terminal ingredients} \label{subsection:analysis-Koopman-space}
We construct the terminal cost and terminal set \emph{in the lifted space} using the Koopman LQR. In the following, we assume \Cref{assum:nonlinear-prop,assum:Koop-obser} hold unless stated otherwise. 
\vspace{-2mm}
\begin{itemize}[leftmargin=1em]
\item \textbf{Terminal cost}:
Let $K$ be the optimal Koopman LQR gain from \cref{eq:optimal-Koopman-control}
and $A_K := A + B K$. We choose matrix $\hat{Q} \succ 0$ such that $\hat{Q} \succeq C^\tr Q C + K^\tr R K$.
We design the terminal cost 
\begin{equation} \label{eq:terminal-cost}
    V_\f(z) := z^\tr \hP  z \le \sigma_\hP \| z\|^2,
\end{equation}
where $\hP \succ 0$ is the solution to the Lyapunov equation $A_K^\tr \hP A_K -\hP +\hat{Q} = 0$ (note that $A_K$ is Shur stable) and $\sigma_\hP$ is its maximum eigenvalue.
\vspace{-2mm}
\item \textbf{Terminal set}: The terminal set is designed as 
\begin{equation} \label{eq:terminal-set}
\mathcal{Z}_\f:=\{z \in \mathbb{R}^{n_z} \ | \ V_\f(z) \le \tau \},
\end{equation}
where $\tau > 0$ is chosen such that $\sqrt{\frac{\tau}{\sigma_R}}\mathcal{B}_{1}  \subseteq \mathcal{U}$. The existence of $\tau$ is guaranteed as $\mathbb{0} \in \mathrm{int}(\mathcal{U})$. 
\end{itemize}

\vspace{-2mm}
With these terminal ingredients, the Koopman LQR controller $u_t = Kz_t= K\Psi(x_t)$ can ensure a sufficient decrease in terminal cost at each step and guarantee that the terminal set $\mathcal{Z}_\f$ is invariant for the nominal Koopman linear model. 
\begin{prop}\longthmtitle{Terminal Controller}
\label{prop:terminal}
    Consider the terminal cost $V_\f$  in~\cref{eq:terminal-cost} and terminal set $\mathcal{Z}_\f$ in~\cref{eq:terminal-set}. For the terminal controller $\kappa_\f(z):= K z$ with $K$ from \cref{eq:optimal-Koopman-control}, we have $\kappa_\f(z)\! \in \! \mathcal{U}$, $\forall z \! \in \! \mathcal{Z}_\f$, and
    \begin{equation}
    \label{eqn:terminal-condi}
     \ V_\f(\bar{z}^+) - V_\f(z) % \le -z^\tr \Pi ^\tr \hat{Q} \Pi z 
     \le -l(Cz, \kappa_\f(z)), \quad \forall z \in \mathcal{Z}_\f,
     \end{equation}
    where $\bar{z}^+ \!=\! Az \!+\! B\kappa_\f(z)$ is the one-step Koopman prediction.      
\end{prop}
This is a standard result in linear MPC, and we provide a proof in \Cref{appendix:proof-terminal-set} for completeness.

\vspace{-3mm}
\subsection{Asymptotic stability for an exact Koopman model}
\vspace{-1mm}
Since the Koopman MPC problem \cref{eq:Koopman-MPC} is designed using the Koopman linear model $z^+ = Az + Bu$, it naturally inherits key properties of standard MPC for linear systems~\citep{rawlings2017model}, such as boundedness, decrease of the value function, and recursive feasibility of the Koopman update.

At time step $t$, given a sequence of control inputs $\mathbf{u} := (u_t, u_{t+1}, \ldots, u_{t+N-1})$ and an initial state $z_t = z$, we denote the nominal Koopman state prediction at time $t+k$ as $\phi(k;z,\mathbf{u}) := z_{t+k}$, where $k \in \mathbb{Z}_{[0,N]}$. We also denote the objective value of \cref{eq:Koopman-MPC} with initial condition $z$ and input sequence $\mathbf{u}$ as $V_N(z, \mathbf{u})$ and its feasible region as $\mathcal{Z}_N\! :=\! \{z \!\in\! \mathbb{R}^{n_z} | \exists \ub \!\in \!\mathcal{U}^N\! \, \textrm{such that}\, \phi(N\!; z,\ub) \!\in\! \mathcal{Z}_\f \ \;\textrm{and} \;\, C\phi(k;z, \mathbf{u}) \in \mathcal{X}, k = 0, \ldots, N-1
\}$, where $\mathcal{U}^N :=\mathcal{U}\times \cdots \times \mathcal{U}$. The following properties are standard in stabilizing linear MPC; see \citep[Ch. 2.4]{rawlings2017model} for details.
\begin{prop}\longthmtitle{Continuity and boundedness}
\label{prop:boundedness}
Consider the Koopman MPC problem \cref{eq:Koopman-MPC} with terminal design $V_\f(\cdot)$ in \cref{eq:terminal-cost} and $\mathcal{Z}_\f$ in  \cref{eq:terminal-set}. Let $V_N^*$ denote its optimal value~function. Then, we~have: 
\vspace{-7mm}
\begin{enumerate}
    \item $\mathbb{0} \in \mathrm{int}(\mathcal{Z}_N)$ and $V_N^*$ is continuous on the interior of its domain $\mathcal{Z}_N$;
    \vspace{-2mm}
    \item $V_N^*$ is bounded with respect to $z, Cz$, and the optimal control sequence $\ub_z^*$: 
  \[
  \begin{aligned}
    \lambda_Q \|Cz\|^2 \le V_N^*(z) \le c_z \|z\|^2, \quad 
    \lambda_R\|\ub_z^*(0)\|^2 \le \lambda_R\|\ub_z^*\|^2 \le V_N^*(z),  \quad \forall z \in \mathcal{Z}_N,
    \end{aligned}
  \]
  where $c_z$ is a positive constant and  $\lambda_Q,
  \lambda_R \in \mathbb{R}$ are the minimum eigenvalues of $Q$ and $R$.
  \end{enumerate}
\end{prop}
\vspace{-5mm}
\begin{prop}\longthmtitle{Recursive feasibility and cost decrease}
\label{prop:rec-feasibility}
Consider the one-step Koopman prediction $\bar{z}^+ = Az + B \kappa(z)$, where $\kappa(z)$ is the Koopman MPC law \cref{eq:Koopman-MPC-law}. Then, we have: 
\vspace{-2mm}
\begin{enumerate}
    \item $\bar{z}^+ \in \mathcal{Z}_N$, \emph{i.e.}, there exists a feasible input sequence to  \cref{eq:Koopman-MPC} with initial state $\bar{z}^+$. One such feasible choice is the one-step shifted input sequence $\ubb := (\ub_z^*(1:N-1), \kappa_\f(\phi(N;z,\ub_z^*)))$;
\vspace{-8mm}
    \item The optimal value function $V_N^*$ decreases at the next Koopman prediction $\bar{z}^+$, \emph{i.e.},
    \begin{equation} \label{eq:decrease-Vn-Koopman}
    V_N^*(\bar{z}^+) \le V_N(\bar{z}^+, \ubb) %\le V_N^*(z) - l(Cz, \kappa(z)) 
    \le V_N^*(z) -\lambda_Q \|Cz\|^2, \quad \forall z \in \mathcal{Z}_N.
    \end{equation}
\end{enumerate}
\end{prop}

\vspace{-4mm}
These properties are closely related to the requirements on a Lyapunov function in \Cref{lemma:LES}. However, we note that the optimal value function $V_N^*$ only has a decrease in terms of $\|Cz\|^2$, instead of $\|z\|^2$. If the Koopman model is exact, the next result shows that the Koopman MPC asymptotically stabilizes the physical state $x$ of the original nonlinear system~\cref{eqn:nonlinear}. 
\begin{thm}
\label{prop:output-stabi}
    Suppose $ e(Cz_t,u_t) \equiv \mathbb{0}$ in \cref{eq:process-error-term} and \Cref{assum:nonlinear-prop,assum:Koop-obser} hold. Consider the Koopman MPC \cref{eq:Koopman-MPC} with terminal design $V_\f(\cdot)$ in \cref{eq:terminal-cost} and $\mathcal{Z}_\f$ in \cref{eq:terminal-set}.
    Then,
    \vspace{-2mm}
    \begin{enumerate}
           \item  The entire feasible region of the Koopman MPC \cref{eq:Koopman-MPC}, $\mathcal{X}_N:= \{x \in \mathbb{R}^n \ | \  \Psi(x) \in \mathcal{Z}_N\}$, is an ROA of the closed-loop system $x_{t+1} = f(x_t, \kappa(\Psi(x_t)))$;
           \vspace{-2mm}
   \item The closed-loop system is also locally asymptotically stable at the origin.
    \end{enumerate}
\end{thm}
\vspace{-2mm}

We prove that $\mathcal{X}_N$ is a region of attraction via showing it is invariant and the running cost of MPC with the optimal Koopman control law is finite. For the Lyapunov stability part, the proof of \Cref{prop:output-stabi} is similar to that in \Cref{proposition:global-stable}. The detailed proof is shown in \Cref{Appendix:output-stabi}.

Analogously to \Cref{proposition:global-stable}, \Cref{assum:nonlinear-prop,assum:Koop-obser} do not guarantee \emph{local exponential stability} of the closed-loop~Koopman MPC system, and additional regularity of the lifting $\Psi$~is~required.
\vspace{-2mm}
\begin{rem}\longthmtitle{Enlarging ROA via Koopman MPC}
With state and input constraints, the MPC $N$-step feasible set
$
\mathcal{X}_N\! :=\! \{x \! \in\!\mathbb{R}^n \! \mid \! \Psi(x)\in \mathcal{Z}_N \}
$ 
includes the ROA from the Koopman LQR control law because $\mathcal{Z}_\f \subseteq \mathcal{Z}_N$. Consequently, Koopman-MPC naturally \emph{enlarges~the ROA} of the nonlinear closed-loop system while enforcing constraints. This is in direct analogy with standard stabilizing MPC for linear systems; see \citep[Ch. 2.5]{rawlings2017model}. \hfill $\square$
\end{rem}

\vspace{-4mm}
\subsection{Local exponential stability of stabilizing Koopman MPC}
\label{subsec:exp-stable}
\vspace{-1mm}
We show the Koopman MPC locally exponentially stabilizes the nonlinear system under suitable assumptions. We relax the state constraint $\mathcal{X}$ in \eqref{eq:Koopman-MPC} for simplicity ($\mathcal{Z}_N$ and $\mathcal{X}_N$ change accordingly). Its satisfaction can be addressed through a more involved discussion, which we defer to future work. Similar to \Cref{assum:pred-erro}, we make the following assumption for the prediction~error. 
\vspace{-1mm}
\begin{assum}
\label{assum:error-bound}
     There exists constants $r>0$ and $L>0$, such that the closed-loop one-step prediction error $e_\C(x)$ in~\cref{eq:process-error-term} with the Koopman MPC law~\cref{eq:Koopman-MPC-law} satisfies
     \begin{equation} \label{eq:error-bound}
     \|e_\C(x)\| \le L \|x\|, \quad \forall x \in \mathcal{B}_r.
     \end{equation}
\end{assum}
\vspace{-5mm}

Let $\rho \in (0, \lambda_Q \hat{r}^2]$ where $\hat{r} := \min\{r_\psi, r\}$ with $r_\psi$ and $r$ from \Cref{assum:lift-cd} and \Cref{assum:error-bound} respectively. We choose the sublevel set  
\begin{equation}
\label{eq:S-ROA}
\mathcal{S} := \{x \in \mathcal{X}_N \ | \ V_N^*(\Psi(x)) < \rho\}.
\end{equation}
This set $\mathcal{S}$ is bounded because of $\lambda_Q \|x\|^2 < \rho$ from \Cref{prop:boundedness}, which also implies $\mathcal{S} \subseteq \mathcal{B}_{\hat{r}} \subseteq \mathcal{B}_{r_\psi} \cap \mathcal{B}_r$. Furthermore, $\mathbb{0}$ is an interior point of $\mathcal{S}$ because of $V_N^*(\Phi(\mathbb{0})) = V_N^*(\mathbb{0}) = 0$ and $V_N^*(\Psi(x))$ is continuous (see \Cref{assum:Koop-obser} and \Cref{prop:boundedness}). We will prove that $\mathcal{S}$ is an ROA for the closed-loop system when the Koopman prediction error is sufficiently~small. 

We first prove the one-step feasibility of the Koopman MPC \cref{eq:Koopman-MPC} over the entire $\mathcal{S}$.
\vspace{-1mm}
\begin{lem}
\label{lemma:feasibility}
    Suppose \Cref{assum:nonlinear-prop,assum:Koop-obser,assum:lift-cd,assum:error-bound} hold. Fix  $\rho \in (0, \lambda_Q \hat{r}^2]$. Then, there exists $\delta_1 > 0$ such that, if $L$ in~\cref{eq:error-bound} satisfies $L \le \delta_1$, the Koopman MPC \cref{eq:Koopman-MPC} is feasible at $\Psi(x^+) $ for all $ x \in \mathcal{S}$.
\end{lem}

Similar to \Cref{prop:rec-feasibility}, this result guarantees the feasibility of \cref{eq:Koopman-MPC} for the Koopman lifting of the true physical state $x^+$. The proof is constructive by showing that $\ubb$ in \Cref{prop:rec-feasibility} is a feasible solution with sufficiently small~$\delta_1$. The details are given in \Cref{append:one-step-feasibility}. 

\Cref{lemma:feasibility} guarantees that $V_N^*(\Psi(x^+))$ is well-defined for $x \in \mathcal{S}$. We next prove that $V_N^*(\Psi(x^+)) < V_N^*(\Psi(x))$ when $x \in \mathcal{S} \setminus \{0\}$ and the constant $L$  in \cref{eq:error-bound} is sufficiently small. 
\begin{lem}
\label{lemma:decent}
    Suppose \Cref{assum:nonlinear-prop,assum:Koop-obser,assum:lift-cd,assum:error-bound} hold. Fix  $\rho \in (0, \lambda_Q \hat{r}^2]$. Then,  there exists $c, \delta_2 >0$, such that if $L$ in~\cref{eq:error-bound} satisfies $L < \delta_2 \le \delta_1$, we have
    \[
    V_N^*(\Psi(x^+)) - V_N^*(\Psi(x)) \le -c \|x\|^2, \ \forall x \in \mathcal{S}.
    \]
\end{lem}
\vspace{-6mm}

The proof of this result is similar to the decomposition of the Lyapunov change in \Cref{them:exp-LQR}. The one-step feasibility in \Cref{lemma:feasibility} and the descent property in \Cref{lemma:decent} can hold for any state in $\mathcal{S}$. This leads to the recursive feasibility and local exponential stability of the closed-loop system.
\begin{thm}
\label{them:exp-stable}
    Suppose \Cref{assum:nonlinear-prop,assum:Koop-obser,assum:lift-cd,assum:error-bound} hold. Fix a  $\rho \in (0, \lambda_Q \hat{r}^2]$ and define the sublevel set $\mathcal{S}$ in~\cref{eq:S-ROA}. There exists $\delta > 0$ such that if $L$ in \cref{eq:error-bound} satisfies $L \le \delta$, the closed-loop Koopman MPC system~\cref{eq:closed-loop-Koopman} is locally exponentially stable and the set $\mathcal{S}$ is an ROA of it.  
\end{thm}
\begin{proof}
    Consider the optimal value function $V_N^*(\Psi(\cdot))$ as a Lyapunov candidate. We next prove that 1) the set $\mathcal{S}$ is invariant, and 2) $V_N^*(\Psi(x))$ satisfies the conditions \cref{eq:quad-bounds} and \cref{eq:decrease} over $\mathcal{S}$ in \Cref{lemma:LES}. 
    
    We can choose $\delta < \delta_2 \le \delta_1$, where $\delta_1$ and $\delta_2$ come from \Cref{lemma:feasibility} and \Cref{lemma:decent}. Then, from \Cref{lemma:decent}, we have the following decent property
    \begin{equation}
    \label{eqn:sk-decrease}
    V_N^*(\Psi(x^+)) - V_N^*(\Psi(x)) \le -c \|x\|^2, \ \forall x \in \mathcal{S},
    \end{equation}
    which implies $V_N^*(\Psi(x^+)) \le V_N^*(\Psi(x)) < \rho$ and $x^+ \in \mathcal{S}$. This guarantees that 1) the set $\mathcal{S}$ is invariant and 2) the Koopman MPC \cref{eq:Koopman-MPC} is recursive feasible.

    Furthermore, for any $x$ in $\mathcal{S}$, we have
    \begin{equation}
    \label{eqn:sk-bound}
    \lambda_Q \|x\|^2 \le V_N^*(\Psi(x)) \le c_z \|\Psi(x)\|_2^2 \le c_z L_\psi^2 \|x\|_2^2,
    \end{equation}
    where the last inequality comes from \Cref{assum:lift-cd}. Inequalities \cref{eqn:sk-decrease} and \cref{eqn:sk-bound} confirm that the value function $V_N^*(\Psi(x))$ satisfies \Cref{eq:quad-bounds,eq:decrease} in \Cref{lemma:LES}. This completes the proof. 
\end{proof}

\vspace{-2mm}
The proof of \Cref{them:exp-stable} closely parallels that of \Cref{them:exp-LQR}.  
The theoretical upper bound $\delta$ for the error coefficient $L$ depends on all problem data in the Koopman MPC \cref{eq:Koopman-MPC} as well as the set $\mathcal{S}$. While $L$ can be evaluated \emph{a posteriori} once the Koopman model and $\mathcal S$ are fixed, enforcing $L\le\delta$ during identification (\emph{e.g.}, as an explicit constraint in EDMD) is generally nontrivial. A convenient sufficient surrogate is a Taylor-like condition on the one-step prediction error $e(\tilde x)$, with $\tilde x=\operatorname{col}(x,u)$ as in \cref{eq:process-error-term}: $\lim_{\|\tilde{x}\|\to 0}{\|e(\tilde{x})\|}/{\|\tilde{x}\|} =  \mathbb{0}.$ This ensures that the constant $L$ can be chosen arbitrarily close to zero when decreasing $r$ in \Cref{assum:error-bound}. Thus, by shrinking the neighborhood, one can ensure the bound in \Cref{assum:error-bound} with arbitrarily small $L$, and \Cref{them:exp-stable} applies. 

\vspace{-1.5mm}
\begin{rem}[Comparison with existing Koopman MPC] \label{remark:comparison}
\!\!\!Our proof strategies for \Cref{prop:output-stabi,them:exp-stable} are similar to standard stabilizing MPC for linear systems \citep[Chap 2]{rawlings2017model}. The Koopman MPC \cref{eq:Koopman-MPC} can naturally stabilize the Koopman linear model $z^+ = Az + Bu$. With \Cref{assum:lift-cd} and \Cref{assum:error-bound} on the lifting function and the one-step prediction error, the Koopman MPC \cref{eq:Koopman-MPC} can also exponentially stabilize the original nonlinear system. We here compare with some existing results. \citet{zhang2022robust,mamakoukas2022robust} focused on ensuring constraint satisfaction and their settings are closer to the robust MPC framework. In \citet{bold2024data, schimperna2025data}, a variant of Koopman MPC is designed without terminal ingredients, but based on a cost controllability assumption of the nonlinear system. The closest studies in the literature are \citet{worthmann2024data,de2024koopman}, where the Koopman MPC formulations also include suitable terminal ingredients. In particular,  \citet{worthmann2024data} mainly focused on a Koopman \textit{bilinear} model and its terminal ingredients are constructed based on the original nonlinear system.  In \citet{de2024koopman}, the initial condition is interpolated, which may lead to an unbounded prediction error. Both \citet{worthmann2024data} and \citet{de2024koopman} only show practical asymptotic stability. Instead, our \cref{them:exp-stable} establishes the exponential stability of \sk{}. This result only requires one assumption for the actual nonlinear system in \Cref{assum:nonlinear-prop} and several mild assumptions for the Koopman linear model in \Cref{assum:Koop-obser,assum:lift-cd,assum:pred-erro,assum:error-bound}.
\hfill $\square$
\end{rem}

\vspace{-4mm}
\section{Numerical experiments}
\vspace{-1mm}
\label{sec:Num-exp}
In this section, we illustrate the performance of \sk{} on a standard inverted pendulum \citep{strasser2024koopman,zhang2022robust}. We compare \sk{} with a standard linear MPC (\method{L-MPC}) based on a first-order Taylor linearization at the origin. 

\vspace{-2mm}
\subsection{Experiment setup}
\vspace{-1mm}
We consider an inverted pendulum with dynamics 
\begin{equation}
    \label{eq:inverted-pendulum}
    \begin{bmatrix}
        x_1 \\ x_2
    \end{bmatrix}^+ \!=\! 
    \begin{bmatrix}
        1 & T_s \\ 0 & 1-\frac{bT_s}{ml^2}
    \end{bmatrix}
    \! \begin{bmatrix}
        x_1 \\ x_2
    \end{bmatrix} \!+\! \begin{bmatrix}
        0 \\ \frac{T_s}{ml^2}
    \end{bmatrix} u +\! \begin{bmatrix}
        0 \\ \frac{gT_s}{l} \sin(x_1)
    \end{bmatrix}
\end{equation}
where the parameters are $m=1\,\mathrm{kg}$, $l=1\,\mathrm{m}$, $b=0.2$, $g=9.81\,\mathrm{m/s^2}$, and $T_s=0.02\,\mathrm{s}$.  

A Koopman predictor is learned via EDMD using the dictionary $\Psi(x)=\operatorname{col}(x_1,x_2,\sin x_1)$. Specifically, we generate $200$ trajectories of length $1000$; initial states and inputs are sampled from $[-2,2]\times[-8,8]$ and $[-40,40]$, respectively. The identified model that is projected to satisfy \Cref{assum:Koop-obser,assum:error-bound} and the Taylor linearization at the origin are
\[
\begingroup
    \setlength\arraycolsep{3pt}
\def\arraystretch{0.85} 
\textbf{Koopman:} \ A \!=\! \begin{bmatrix}
    1 & 0.02 & 0 \\ 0 & 0.996 & 0.1962 \\ 0.002 & 0.02 & 0.998
\end{bmatrix}, B \!=\! \begin{bmatrix}
    0 \\ 0.02 \\ 0
\end{bmatrix}, \ \textbf{Taylor:} \ A \!=\! \begin{bmatrix}
    1 & 0.02  \\ 0.1962 &  0.996 
\end{bmatrix}, B \!=\! \begin{bmatrix}
    0 \\ 0.02 
\end{bmatrix}.
\endgroup
\]
Both controllers use the stage cost \cref{eq:stage-cost} with $Q=10I$ and $R=I$, and a prediction horizon $N=20$.

\vspace{-3mm}
\subsection{Closed-loop performance}
\begin{wrapfigure}[10]{r}{0.61\textwidth}
\vspace{-7mm}
\centering
{\includegraphics[width=0.3\textwidth]{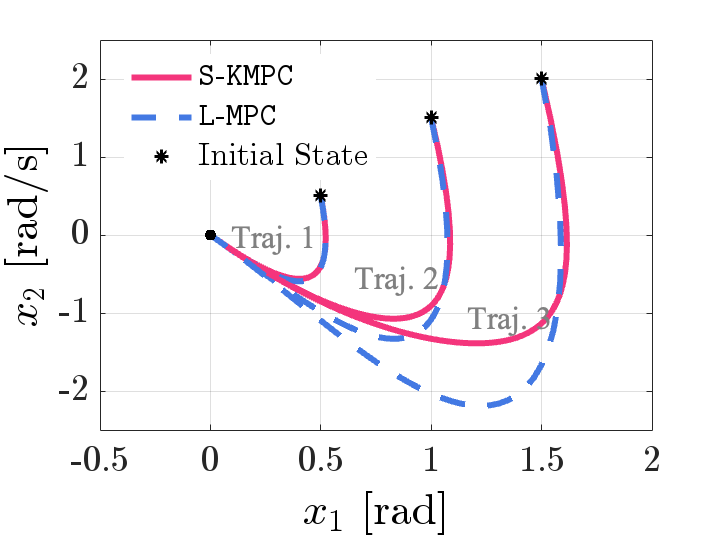} \label{subfig:IP-traj}} \hspace{-3mm}
{\includegraphics[width=0.3\textwidth]{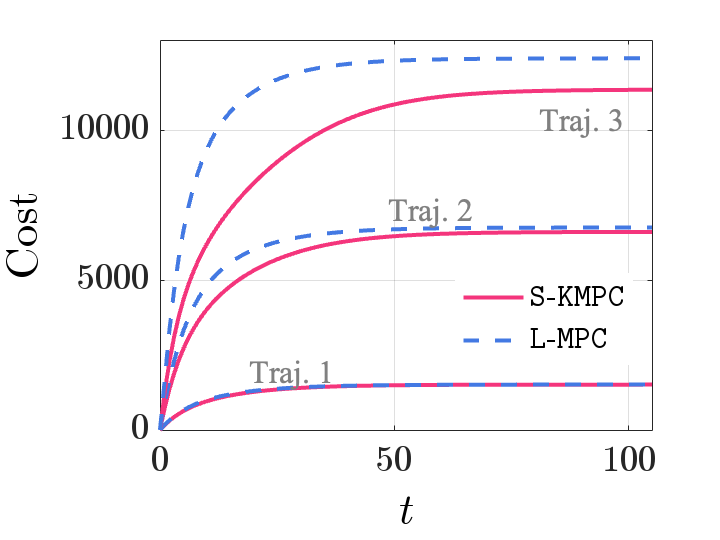}\label{subfig:IP-cost}}
\vspace{-3mm}
\caption{MPC performance for inverted pendulum: approximated Koopman model vs. first-order (Taylor) linearization. Left: closed-loop trajectory; right: accumulated cost.}
\label{fig:IP-MPC}
\end{wrapfigure}

\vspace{-1mm}

We compare \sk{} and \method{L-MPC} on multiple initial conditions. Both controllers employ terminal ingredients designed according to \Cref{subsection:analysis-Koopman-space}. \Cref{fig:IP-MPC} shows phase portraits and accumulated costs. For all initial states, both controllers can stabilize the pendulum, and the accumulated costs converge to steady values. However, as the initial state moves farther from the origin, \method{L-MPC} exhibits larger state deviation than \sk{} and consequently higher accumulated cost. Relative to \sk{}, the cost increase of \method{L-MPC} is approximately $0.13\%$, $2.2\%$, and $9.2\%$ for trajectories 1-3, respectively, growing with the initial distance to the origin.

The rationale for this behavior is intuitive. For small deviations, both the Koopman model and the Taylor model are relatively accurate, leading to similar performance. As the initial state moves away from the linearization point, the Taylor predictor incurs larger modeling error, degrading the performance \method{L-MPC}. In principle, the EDMD-based Koopman predictor may remain reliable over a larger region, and thus \sk{} maintains better trajectories in the transient and lower accumulated costs. These results are consistent with the simulation in \Cref{example:exact-linearization}, cf. \Cref{fig:LQR}.

\vspace{-3mm}
\section{Conclusions}
\label{sec:conclu}
\vspace{-1mm}
We have revisited the local exponential stability of Koopman MPC. In particular, we have presented a stabilizing Koopman MPC variant, where the terminal cost and terminal set are designed based on Koopman LQR. We have shown that the nonlinear closed-loop is locally asymptotically stable when the lifting function and one-step Koopman prediction error are both Lipschitz at the origin, with the Lipschitz constants being small. Numerical simulations confirm the superior performance of the \sk{}, which has a faster convergence rate and a lower accumulated cost. Some future directions include developing data-driven identification methods that can guarantee the satisfaction of the required assumptions, and validating the performance of \sk{} in practical nonlinear systems (\emph{e.g.}, mixed traffic systems \citep{shang2024decentralized} and robotic systems \citep{haggerty2023control}).

%%%%%%%%%%%%%%%%%%%%%%%%%%%%%%%%%%%%%%%%%%%%%%%%%%%%%%%%%%%%%%%%%%%%%%%%%%%%%%%%%%%%%%%%%%%%%%%%%%%%%%%%%%%%%%%%%%%%%%%%%%%%%%%%%%%%%%%%%%%%

\acks{This work is supported by NSF CMMI Award 2320697, NSF CAREER Award 2340713, and ONR Award N00014-23-1-2353.}

\bibliography{l4dc2026_ArXiv}

\newpage
\appendix

\section{Technical proofs}
We here present technical proofs for \Cref{lemma:LES} and the results in \Cref{sec:LQR-stab,section:analysis-original-state}.

\subsection{Proof of \Cref{lemma:LES}}
\label{Appendix:LES}
Following \Cref{def:exp-stab}, we here demonstrate that the Lyapunov function $V$ satisfying the conditions in \Cref{lemma:LES} leads to local exponential stability. 

Let us recall standard definitions for an invariant set and region of attraction.

\begin{deff}[Positive invariant set and region of attraction \citet{rawlings2017model}] Consider the autonomous system $x_{t+1} = g(x_t)$. 
\begin{enumerate}
    \item A set $\mathcal{D}$ is positive invariant if $x_t \in \mathcal{D}$ implies $x_{t+1} \in \mathcal{D}$.
    \item A set $\mathcal{D}$ of initial states $x_0$ with $\lim_{t\to \infty} x_t = \mathbb{0}$ is a region of attraction for the origin $\mathbb{0}$.
\end{enumerate}
\end{deff}

\textbf{Proof of  \Cref{lemma:LES}:} 
From the condition \cref{eq:decrease}, we can obtain
\[
\begin{aligned}
V(x_{t+1}) &\le V(x_t) - \alpha_3\|x_t\|^2 \\
&\le (1-\alpha_3/\alpha_2) V(x_t), \quad \forall x_t \in \mathcal{D},
\end{aligned}
\]
which implies $V(x_t) \le (1-\alpha_3/\alpha_2)^t V(x_0)$. Since the set $\mathcal{D}$ is invariant, the function value $V(x_{t+1})$ is always well-defined when $x_t \in \mathcal{D}$. 

Then, using condition \cref{eq:quad-bounds}, we can write 
\[
\begin{aligned}
\|x_t\| &\le \sqrt{V(x_t)/\alpha_1} \\
&\le \sqrt{(1-\alpha_3/\alpha_2)^t / \alpha_1 V(x_0)} \\
&\le \sqrt{\alpha_2/\alpha_1(1-\alpha_3/\alpha_2)^t} \|x_0\|, \quad \forall x_0 \in \mathcal{D}.
\end{aligned}
\]
This implies that states in $\mathcal{D}$ converge to the origin exponentially and $\mathcal{D}$ is an ROA. 

Let $c:= \sqrt{\alpha_2/\alpha_1}$ and $\rho := \sqrt{1-\alpha_3/\alpha_2}$. We note that there exists a neighborhood $\mathcal{N} \subseteq \mathcal{D}$ as $\mathbb{0} \in \mathrm{int}(\mathcal{D})$. Thus, we have $\|x_t\| \le c\rho^t \|x_0\|, \forall x_0 \in \mathcal{N}, t \in \mathbb{Z}_{\ge 0}$, which completes the proof.

\subsection{Proof of \Cref{lemma:optimal-Koopman-control}}
\label{Appendix:optimal-Koopman}
Thanks to the exact Koopman linear embedding, we can equivalently rewrite the nonlinear dynamics in \cref{eq:optimal-control} in the Koopman-lifted space as (we restate \cref{eqn:LQR-KL} here)
\[
\begin{aligned}
\min_{\mathbf{u}_\infty} & \quad \sum_{k = 0}^{\infty} l(Cz_{t+k}, u_{t+k}) \\
\text{subject to}& \quad z_{t+k+1} = Az_{t+k} + Bu_{t+k}, \\
& \quad z_t = z, k \in\mathbb{Z}_{\ge 0},
\end{aligned}
\]
where $z = \Psi(x)$ and we have $x_{t+k} = Cz_{t+k}, k \in \mathbb{Z}_{\ge 0}$. For a fixed $x \in \mathbb{R}^n$, the cost values in \cref{eq:optimal-control} and \cref{eqn:LQR-KL} are the same given the same input sequence $\mathbf{u}_\infty$.  

Since \cref{eqn:LQR-KL} can be viewed as a standard LQR problem with an initial condition $z_t = \Psi(x_t)$ under \Cref{assum:Koop-obser}, the controller \cref{eq:optimal-Koopman-control} is an optimal feedback policy to \cref{eqn:LQR-KL}. Therefore, \cref{eq:optimal-Koopman-control} is also a globally optimal feedback controller to the original problem \cref{eq:optimal-control}.

\subsection{Proof of \Cref{proposition:global-stable}}
\label{Appendix:global-stable}
We need to show that: 1) the system is globally attractive, that is, $\lim_{t\to \infty} x_t = 0, \; \forall x_0 \in \mathbb{R}^n$ and 2) the system is locally Lyapunov stable, that is, for any $\epsilon >0$, there exists $\delta >0$, such that $\|x_0\| < \delta$ implies $\|x_t\| < \epsilon$, for all $t \in \mathbb{Z}_{\ge 0}$. 

We recall the associated optimal value function for \cref{eqn:LQR-KL} is $V_{\infty}^*(z) = \|z\|_P^2:=z^\tr P z$ and we have the following lower and upper bounds: % \citep{}.
\begin{equation}
\label{eqn:LQR-value}
\lambda_Q \|Cz\|^2 \le V_{\infty}^*(z) \le \sigma_P \|z\|^2,
\end{equation}
where $\lambda_Q$ is the minimum eigenvalue of $Q$, and $\sigma_P$ is the maximum eigenvalue of~$P$ (see \eqref{eqn:Riccati}).
    
The globally attractive property is easy to show. Since the optimal LQR gain can stabilize the exact Koopman linear model, we have $\lim_{t \rightarrow \infty} z_t = 0, \forall z_0 \in \mathbb{R}^{n_z}$ which implies $\lim_{t\rightarrow \infty} \Psi(x_t) = 0, \forall x_0 \in \mathbb{R}^n$. As $x_t = C \Psi(x_t)$ by \cref{assum:Koop-obser}, the physical state $x_t$ of the nonlinear system converges to the origin asymptotically. 
    
We next show the local stability. For any fixed $\epsilon>0$, let $\rho = \lambda_Q \epsilon^2$. Consider the set 
$$
\mathcal{S} := \{x \in \mathbb{R}^n \ | \ V_{\infty}^*(\Psi(x)) < \rho \}.
$$ 
Using the Riccati equation, for any $x \in \mathcal{S}$, we have 
    \[
    V_{\infty}^*(\Psi(x^+)) - V_{\infty}^*(\Psi(x)) 
    = \ -\Psi(x)^\tr (C^\tr Q C + K^\tr R K)\Psi(x)\le 0. 
    \]
This implies that $ V_{\infty}^*(\Psi(x^+)) \le V_{\infty}^*(\Psi(x)) < \rho $, and thus we have $x^+ \in \mathcal{S}$. Therefore, $\mathcal{S}$ is an invariant set. 
    
On the other hand, from \cref{eqn:LQR-value}, we have 
    \[
    \lambda_Q \|x\|^2 \le V_{\infty}^*(\Psi(x)) < \rho \qquad \Rightarrow \qquad \|x\| < \epsilon, \forall x \in \mathcal{S},
    \]
meaning that $\mathcal{S} \subseteq \mathcal{B}_\epsilon$. Since $V_{\infty}^*(\Psi(x))$ is continuous, we know that the set $\mathcal{S}$ is open. In addition, we have $V(\Psi(\mathbb{0})) = V(\mathbb{0}) = \mathbb{0}$, implying $\mathbb{0} \in \mathcal{S}$. 
    
Thus, there exists a neighbor $\mathcal{N} \subset \mathcal{S}$ of the origin such that any trajectory starting from $\mathcal{N}$ remains in $\mathcal{S}$. It is also contained in $\mathcal{B}_\epsilon$. This establishes the stability in the sense of Lyapunov. We now complete the proof. 

\subsection{Proof of \Cref{them:exp-LQR}}
\label{Appendix:exp-LQR}
Consider a Lyapunov candidate 
\begin{equation} \label{eq:Lyapunov-candidate-LQR}   
\tilde{V}(x) =  V_{\infty}^*(\Psi(x)) :=\Psi(x)^\tr P \Psi(x), 
\end{equation}
where $P \succ 0$ is the unique solution to \cref{eqn:Riccati}. 
We show that $\tilde{V}$ satisfies \cref{eq:quad-bounds,eq:decrease} in \Cref{lemma:LES} over an invariant set~$S$. 

Let $\rho \in (0, \lambda_Q \hat{r}^2]$, where $\hat{r} := \min\{r_\psi, r\}$ with $r_\psi$ and $r$ from \Cref{assum:lift-cd} and \Cref{assum:pred-erro} respectively.  We choose the sublevel set 
$$\mathcal{S}:= \{x \in \mathbb{R}^n \ | \ \tilde{V}(x) < \rho\}.$$
Note that $\mathbb{0} \in \mathcal{S}$ and $\mathcal{S}$ is open as $\tilde{V}$ is continuous. We will prove that $S$ is bounded and invariant.

We first verify the lower bound~in~\cref{eq:quad-bounds}. 
This directly comes from \cref{eqn:LQR-value} as we have $x\! =\! C\Psi(x)$ by~\Cref{assum:Koop-obser}: 
\begin{equation} \label{eq:lower-bounds-LQR}
   \lambda_Q \|x\|^2 \le V_{\infty}^*(\Psi(x))=\tilde{V}(x), \quad \forall x \in \mathcal{S}. 
\end{equation}
By our choice $\rho \le {\lambda_Q} \hat{r}^2$, we have
$
\lambda_Q \|x\|^2 \le \tilde{V}(x) < \rho \Rightarrow \|x\| < \hat{r}$, for all $x \in \mathcal{S}.
$
Thus, we have $\mathcal{S} \subseteq \mathcal{B}_{\hat{r}}$, which is bounded.

The upper bound in~\cref{eq:quad-bounds} is ensured by \Cref{assum:lift-cd}. As $\mathcal{S} \subseteq \mathcal{B}_{\hat{r}} \subseteq \mathcal{B}_{r_\psi}$, we have
\begin{equation} \label{eq:upper-bounds-LQR}
   \tilde{V}(x) \leq \sigma_P \|\Psi(x)\|^2 \le \sigma_P L_\psi^2 \|x\|^2, \quad \forall x \in \mathcal{S}.
\end{equation}
Combining \cref{eq:lower-bounds-LQR} with \cref{eq:upper-bounds-LQR}, our Lyapunov candidate $\tilde{V}$ satisfies \cref{eq:quad-bounds} with $\alpha_1 = \lambda_Q$ and $\alpha_2 = \sigma_P L_\psi^2$. 

We next verify the decrease condition \cref{eq:decrease}. We see that
\begin{equation}
\begin{aligned}
    \tilde{V}(x^+)\! -\! \tilde{V}(x)
    &\!=\! V^*_\infty(\Psi(x^+)) \!-\! V^*_\infty(\Psi(x)) \\
    &\!=\!\underbrace{V^*_\infty(\Psi(x^+)) \!-\! V^*_\infty(\bar{z}^+)}_{\text{Koopman error}} \!+\! \underbrace{V^*_\infty(\bar{z}^+)\!-\!  V^*_\infty(z)}_{\text{Koopman decrease}}, \label{eq:decrease-decomposition}
\end{aligned}
\end{equation}
where we denote $z$ as $\Psi(x)$ and $\bar{z}^+ = (A+BK)z$ is the one-step ahead Koopman prediction. The Koopman decrease term in \cref{eq:decrease-decomposition} can be bounded as  
\begin{equation}
\label{eqn:LQR-Vde-nom}
\begin{aligned}
 V_{\infty}^*(\bar{z}^+) -  V_{\infty}^*(z) 
 &=-z^\tr (C^\tr Q C + K^\tr R K)z \\
 &\le -\lambda_Q \|x\|^2.
\end{aligned}
\end{equation}
where we have used the Riccati equation \cref{eqn:Riccati} and $x = Cz$. 

Meanwhile, consider the Koopman error term in \cref{eq:decrease-decomposition}, for any $x \in \mathcal{S}$, we have 
\begin{equation}
\label{eqn:bound-variation}
\begin{aligned}
\|V_\infty^*(\Psi(x^+)) - V_\infty^*(\bar{z}^+)\| 
& = \|\bar{e}_\C(x)^\tr P\bar{e}_\C(x) 
+ 2  \bar{e}_\C(x)^\tr P\bar{z}^{+}\|  \\
& \le  \sigma_P \|\bar{e}_\C(x)\|^2 \! +\! 2\sigma_P \|\bar{z}^+\| \|\bar{e}_\C(x)\| \\
&\leq \; (\sigma_P L + 2 \bar{\sigma} L_\psi \sigma_P)L\|x\|^2,
\end{aligned}
\end{equation}
where the last inequality uses \Cref{assum:lift-cd,assum:pred-erro} (note $\mathcal{S} \subseteq \mathcal{B}_{\hat{r}} \subseteq \mathcal{B}_{r_\psi} \cap \mathcal{B}_{r}$), and $\bar{\sigma}$ denotes the maximum singular value of $A + BK$. We can now find $\delta > 0$ such that,  when $L < \delta$, we have 
$$
\alpha_3:= \lambda_Q - (\sigma_P L + 2 \bar{\sigma} L_\psi\sigma_P)L> 0.
$$
Substituting \cref{eqn:LQR-Vde-nom} and \cref{eqn:bound-variation} into \cref{eq:decrease-decomposition}, we have  
\begin{equation}
\label{eqn:LQR-Vde-act}
\begin{aligned}
\tilde{V}(x^+) - \tilde{V}(x) \le -\alpha_3 \|x\|^2, \quad \forall x \in \mathcal{S},
\end{aligned}
\end{equation}
which implies $x^+ \in \mathcal{S}$ and $\mathcal{S}$ is invariant. Since the Lyapunov candidate \cref{eq:Lyapunov-candidate-LQR} satisfies the conditions in \Cref{lemma:LES}, the result follows.

\begin{rem} \label{remark:cost-decrease}
A key step is the decomposition of the one-step Lyapunov change in \cref{eq:decrease-decomposition}. This is a standard strategy in showing the inherent robustness of MPC that uses a nominal model with the modelling error or external disturbance \citep{allan2017inherent}. 

Specifically, we first demonstrate the MPC controller can stabilize the nominal model (Koopman decrease) and then treat the modelling error as a perturbation (Koopman error). 
The \emph{Koopman decrease} term is certified by the Riccati identity together with the state inclusion $x =Cz$, yielding~\cref{eqn:LQR-Vde-nom}. The Koopman error depends on both the lifting function and the one-step prediction error, thus, we make \Cref{assum:lift-cd,assum:pred-erro} on them to ensure the Koopman error is sufficiently small. \Cref{assum:lift-cd} (local bound of the lifting function) is used to ensure $\|\Psi(x)\|$ is bounded by the actual state $x$ locally, which (i) gives the \emph{upper quadratic bound} on $\tilde V(x)$ and (ii) ensures $\|\bar z^{+}\|=\|(A{+}BK)\Psi(x)\|\le \bar\sigma L_\psi\|x\|$.
\Cref{assum:pred-erro} (local closed-loop prediction accuracy) bounds the \emph{Koopman prediction-error} term as in \cref{eqn:bound-variation}, which can be made strictly smaller than the Koopman decrease by choosing $L$ sufficiently small.
Thus $L<\delta$ ensures the quadratic decay in~\cref{eqn:LQR-Vde-act} and, in turn, local exponential stability by \Cref{lemma:LES}. 
\hfill$\square$
\end{rem}

\subsection{Proof of \Cref{prop:terminal}} \label{appendix:proof-terminal-set}
We recall that, from the construction of the terminal cost and terminal set, we have
\begin{subequations}
\begin{equation}
\label{eqn:state-stage-cost}
\hat{Q} \succeq C^\tr Q C + K^\tr R K,
\end{equation} 
\begin{equation}
\label{eqn:lyap-eqn}
A_K^\tr \hP A_K -\hP +\hat{Q} = 0 .
\end{equation}
\end{subequations}
We first show the inequality \cref{eqn:terminal-condi} is satisfied with the proposed terminal controller. We can write 
    \[
    \begin{aligned}
    V_\f(\bar{z}^+) \!-\! V_\f(z) &\! =\! z^\tr(A_K^\tr \hP A_K - \hP)z\! =\! -z^\tr \hat{Q} z \\
    & \le \!-z^\tr(C^\tr Q C + K^\tr R K) z\\
    & =\! -l(Cz, \kappa_\f(z)),
    \end{aligned}
    \]
    where the second equality and the third inequality come from~\cref{eqn:lyap-eqn} and \cref{eqn:state-stage-cost}, respectively. 
    
    As $l(Cz, \kappa_\f(z)) \ge 0$, we obtain $$V_\f(\bar{z}^+) \le V_\f(z) \le \tau.$$ Thus, $\bar{z}^+ \in \mathcal{Z}_\f$ and the terminal set is control invariant.

    We then present that the input given by the terminal controller satisfies the input constraint. From the inequality~\cref{eqn:terminal-condi} and the designed terminal set \cref{eq:terminal-set}, we have
\begin{align}
V_\f(\bar{z}^+) + l(Cz, \kappa_\f(z)) 
&= \ V_\f(\bar{z}^+) + z^\tr C^\tr Q Cz + \kappa_\f(z)^\tr R \kappa_\f(z)  \nonumber\\
&\le  \ V_\f(z) \le \tau, \ \forall z \in \mathcal{Z}_\f. \nonumber 
\end{align}
This implies that 
$$
\|\kappa_\f(z)\| \le \sqrt{\frac{\tau}{\sigma_R}}, \ \forall z \in \mathcal{Z}_\f, 
$$
where $\sigma_R$ is the maximum eigenvalue of $R$. Since $\sqrt{\frac{\tau}{\sigma_R}} \mathcal{B}_1 \subseteq \mathcal{U}$ from the construction, we have $\kappa_\f(z) \in \mathcal{U}, \forall z \in \mathcal{Z}_\f$. This completes the proof.

\subsection{Proof of \Cref{prop:output-stabi}}
\label{Appendix:output-stabi}
We first establish $\mathcal{X}_N$ is an ROA, \emph{i.e.}, the Koopman~MPC problem \cref{eq:Koopman-MPC} is recursively feasible and the resulting closed-loop states converge to zero asymptotically from any initial state $x_t \!=\! x\! \in \! \mathcal{X}_N$.

Since the Koopman linear model is exact by assumption, we have 
$$x^+ = C\bar{z}^+.
$$ Meanwhile, we know $\bar{z}^+ \in \mathcal{Z}_N$ by \Cref{prop:rec-feasibility}. Thus, the Koopman MPC problem \cref{eq:Koopman-MPC} is feasible with the initial state $x^+$. We can then denote the resulting optimal Koopman control sequence as 
$$
\mathbf{u}^{\textrm{opt}} := (\kappa(\Psi(x_t)), \kappa(\Psi(x_{t+1})), \ldots, \kappa(\Psi(x_{t+\infty}))).
$$
   
We prove $\lim_{t\rightarrow\infty}x_t = 0$ by showing the running MPC cost that $\sum_{k=0}^{\infty} \|x_{t+k}\|_Q^2+\|\kappa(\Psi(x_{t+k}))\|_R^2$ is finite (recall that $Q, R$ are positive definite). From \cref{eq:decrease-Vn-Koopman}, we have 
\[
     \ V_N^*(\Psi(x_t)) \! \ge \! V_N^*(\Psi(x_{t+1}))\! +\! \|x_t\|_Q^2\! +\! \|\kappa(\Psi(x_t))\|_R^2 \ge  \ %V_N^*(\Psi(x_{t+\infty+1})) \!+\! 
    \sum_{k=0}^{\infty} \|x_{t+k}\|_Q^2\!+\!\|\kappa(\Psi(x_{t+k}))\|_R^2.
\]
As $V_N^*(\Psi(x_t))$ is upper bounded, cf. \Cref{prop:boundedness}, we know $\sum_{k=0}^{\infty} \|x_{t+k}\|_Q^2\!+\!\|\kappa(\Psi(x_{t+k}))\|_R^2$ is finite. Thus, we have $\lim_{t\rightarrow\infty}x_t = 0$, i.e., all states in $\mathcal{X}_N$ converge to the origin asymptotically. 

We then prove the closed-loop system is locally asymptotically stable. This is equivalent to showing that 1) the system is locally attractive and 2) stable in the sense of Lyapunov.~Point 1) is obviously true as $\mathcal{X}_N$ is an ROA and $\mathbb{0} \in \mathrm{int}(\mathcal{X}_N)$. 
    
For point 2), fix any $\epsilon >0$, let $\rho = \lambda_Q \epsilon^2$, and consider $\mathcal{S} := \{x \in \mathcal{X}_N \ | \ V_N^*(\Psi(x)) < \rho \}$. From \cref{eq:decrease-Vn-Koopman}, we have 
\[
     V_N^*(\Psi(x^+))\le V_N^*(\Psi(x)) <  \rho \Rightarrow x^+ \in \mathcal{S}, \ \forall x \in \mathcal{S}.
\]
Thus, $\mathcal{S}$ is an invariant set. From the boundedness property in \Cref{prop:boundedness}, we have 
\[
    \lambda_Q \|x\|^2 \le V_N^*(\Psi(x)) < \rho \Rightarrow \|x\| < \epsilon, \ \forall x \in \mathcal{S},
\]
indicating $\mathcal{S} \subseteq \mathcal{B}_\epsilon$. Since $V_N^*(\Psi(x))$ is  continuous and $\mathbb{0} \in \mathcal{S}$, the origin is an interior point of $\mathcal{S}$ with the associated neighborhood $\mathcal{N} \subseteq \mathcal{S}$. Thus, any trajectory starting in $\mathcal{N}$ remains in $\mathcal{S} \subseteq \mathcal{B}_\epsilon$. This completes the proof.

\subsection{Proof of \Cref{lemma:feasibility}}
\label{append:one-step-feasibility}
We show that $\bar{\mathbf{u}}$ (see the construction in \Cref{prop:rec-feasibility}) is a feasible solution for the initial state $z^+ = \Psi(x^+)$, which guarantees the one-step feasibility. We need to prove $\phi(N; z^+, \ubb) \in \mathcal{Z}_\f$, which is equivalent to show $V_\f(\phi(N;z^+, \ubb)) \le \tau$. Thus, our goal in this part is to derive an upper bound for $V_\f(\phi(N;z^+, \ubb))$. The key idea for the derivation is to first bound the difference between $V_\f(\phi(N; z^+, \ubb))$ and $V_\f(\phi(N;\bar{z}^+,\ubb))$, then derive an upper bound for  $V_\f(\phi(N;\bar{z}^+,\ubb))$, and finally combine both to obtain the upper bound of $V_\f(\phi(N;z^+, \ubb))$.

We first derive upper bounds for the optimal control input $\ubb_z^*(0)$, the input sequence $\ubb,$ and the predicted state $\bar{z}^+$ of the nominal Koopman linear embedding. We will utilize these bounds later in bounding the difference between $V_\f(\phi(N;z^+, \ubb))$ and $V_\f(\phi(N;\bar{z}^+, \ubb))$. Using the bound of $z$ (\emph{i.e.}, $\|z\| = \|\Psi(x)\| \le L_\psi \|x\|$ from \Cref{assum:lift-cd}) and the property of Koopman MPC (see \Cref{prop:boundedness} and \Cref{prop:rec-feasibility}), the upper bound of $\ub_z^*(0)$ and $\ubb$ can be derived as 
\begin{gather}
        \|\ub_z^*(0)\|^2 \le  \frac{V_N^*(z)}{\lambda_R} \le \frac{c_z\|z\|^2}{\lambda_R} \le c_1 \|x\|^2, \label{eqn:inputOpt-bound} \\
        \|\ubb\|^2  \! \le \! \frac{V_N(\bar{z}^+, \ubb)}{\lambda_R} \! \le \! \frac{V_N^*(z)}{\lambda_R} \! \le \! \frac{c_z\|z\|^2}{\lambda_R} \! \le \! c_1 \|x\|^2, \label{eqn:inputSeq-bound}
\end{gather}
where $c_1 = \frac{c_z L^2_{\psi}}{\lambda_R} \in \mathbb{R}$ is a finite constant. Combining both the bound of $z$ from \Cref{assum:lift-cd} and the bound of $\ub_z^*(0)$~\cref{eqn:inputOpt-bound}, we obtain the bound of $\|\bar{z}^+\|$, that is:
\[
    \begin{aligned}
    \|\bar{z}^+\| & = \|Az + B\ub_z^*(0)\|\\ &\le \|A\|\|z\| + \|B\|\|\ub_z^*(0)\| \\
    & \le \sigma_A L_\psi \|x\| + \sigma_B \sqrt{c_1} \|x\| = c_2 \|x\|,
    \end{aligned}
\]
where $\sigma_A$, $\sigma_B$ are maximum singular values of $A, B$, respectively and $c_2 = \sigma_A L_{\psi}+ \sigma_B\sqrt{c_1}$.

We then bound the difference between $V_\f(\phi(N;z^+, \ubb))$ and $V_\f(\phi(N;\bar{z}^+, \ubb))$. We can write the terminal state $\phi(N;z, \ub)$ with given $z$ and $\mathbf{u}$ as 
    \[
    \phi(N;z, \ub) = \bar{O}_N z + \bar{T}_N \ub,
    \]
where, following the propagation of the nominal Koopman linear embedding, we have
    \[
    \bar{O}_N = A^N, \quad \bar{T}_N = \begin{bmatrix}
     A^{N-1}B & A^{N-2}B & \ldots & B   
    \end{bmatrix}.
    \]
Thus, we can derive the difference between $V_\f(\phi(N;z^+, \ubb))$ and $V_\f(\phi(N;\bar{z}^+, \ubb))$ as 
\begin{equation}
\label{eq:diff-final-state}
\begin{aligned}
    & \ \|V_\f(\phi(N;z^+, \ubb))-V_\f(\phi(N;\bar{z}^+, \ubb))\| \\ 
    = & \  \|(\bar{O}_N z^+ + \bar{T}_N \ubb)^\tr \hP (\bar{O}_N z^+ + \bar{T}_N \ubb) - (\bar{O}_N \bar{z}^+ + \bar{T}_N \ubb)^\tr \hP (\bar{O}_N \bar{z}^+ + \bar{T}_N \ubb)\|  \\
    \le & \ 2\sigma_1 \|\bar{z}^+\|\|z^+\!-\!\bar{z}^+\| \! +\! \sigma_1 \|z^+ \!-\! \bar{z}^+\|^2 \! 
     +\! 2 \sigma_2 \|\ubb\| \|z^+ \!-\!\bar{z}^+\| \\ 
    \le & \  2\sigma_1 c_2 L \|x\|^2 \!+\! \sigma_1 L^2 \|x\|^2 \!+\! 2 \sigma_2 \sqrt{c_1}L \|x\|^2 \! = \!  c_3 \|x\|^2\!,
\end{aligned}
\end{equation}
in which $\sigma_1, \sigma_2$ are maximum singular values of $\bar{O}_N^\tr \hP \bar{O}_N$ and $\bar{O}_N^\tr \hP \bar{T}_N$ and we have $c_3 = \sigma_1 L^2+  2(\sigma_1 c_2  + \sigma_2 \sqrt{c_1})L$. We note that the difference between $z^+$ and $\bar{z}^+$ is the closed-loop one-step prediction error in \Cref{assum:error-bound}: 
$$
    \|z^+\! -\! \bar{z}^+\| \! = \! \|\Psi(f(x, \ub_z^*(0)))\!-\! A z \!-\! B \ub_z^*(0) \| \!=\! \|e_\C(x)\| \le L \|x\|.
    $$

We next derive the upper bound of $V_\f(\phi(N;\bar{z}^+, \bar{\mathbf{u}}))$. We consider two cases: 1) $0 \le V_\f(\phi(N;z,\ub_z^*))$ $<$ $\frac{\tau}{2}$; 2) $\frac{\tau}{2} \le V_\f(\phi(N;z,\ub_z^*))\le \tau$. For case I, we have
    \[
    \begin{aligned}
    & V_{\f}(\phi(N; \bar{z}^+, \ubb))-V_{\f}(\phi(N; z, \ub_z^*)) \le 0 \\
    \Rightarrow \ & V_{\f}(\phi(N; \bar{z}^+, \ubb)) \le \frac{\tau}{2}.
    \end{aligned}
    \]
For case II, from the design of the terminal cost (see the Lyapunov equation of $\hP$ in \cref{eq:terminal-cost}) and its upper bound $V_\f(z) \le \sigma_{\hat{P}} \| z\|^2$, we have
    \[
    \begin{aligned}
    & V_{\f}(\phi(N; \bar{z}^+, \ubb))-V_{\f}(\phi(N; z, \ub_z^*)) \\
    =\! &\! -\!\phi(N; z, \ub_z^*)^\tr  \hat{Q} \phi(N; z, \ub_z^*) \! \\
    \le &\! -\! \lambda_{\hat{Q}} \| \phi(N; z, \ub_z^*)\|^2 
    \le -\frac{\lambda_{\hat{Q}}\tau}{2\sigma_{\hP}},
    \end{aligned}
    \]
where $\lambda_{\hat{Q}} > 0$ is the minimum eigenvalue of $\hat{Q} \succ 0$. Thus, we have
    \begin{equation}
    \label{eqn:bound-linearized}
    V_\f(\phi(N; \bar{z}^+, \ubb)) \le \tau -\gamma,
    \end{equation}
where $\gamma = \min\{\frac{\tau}{2}, \frac{\lambda_{\hat{Q}}\tau}{2\sigma_{\hP}}\} > 0$. 
    
Combing~\cref{eq:diff-final-state} and~\cref{eqn:bound-linearized}, we finally have
    \[
    \begin{aligned}
    V_\f(\phi(N;z^+, \ubb))  &\le V_\f(\phi(N;\bar{z}^+, \ubb)) + c_3 \|x\|^2 \\
    &\le  \tau -\gamma + c_3\|x\|^2.
    \end{aligned}
    \]
We can consider $c_3$ as a function of $L$ such that $c_3 =\alpha_1(L) := \sigma_1L^2 + 2(\sigma_1 c_2 + \sigma_2 \sqrt{c_1})L \in \mathcal{K}$. Thus, if $L \le \delta_1 := \alpha_1^{-1}(\frac{\gamma}{c_x})$ where $c_x = \sup_{x \in \mathcal{S}} \|x\|^2$, we have $c_3 \|x\|^2 \le \gamma$ which means $V_\f(\phi(N;z^+, \ubb)) \le \tau$.

\subsection{Proof for \Cref{lemma:decent}} 
\label{append:decent}
    Our goal in this part is to obtain an upper bound of the difference between $V_N^*(\Psi(x^+))$ and $V_N^*(\Psi(x))$ (\emph{i.e.}, $V_N^*(z^+)$ and $V_N^*(z)$) and present it is a negative quadratic function under sufficiently small $\delta_2$. From the descent property of the Koopman MPC (see \Cref{prop:rec-feasibility}), we can obtain an upper bound for the difference between $V_N(\bar{z}^+, \ubb)$ and $V_N^*(z)$. We can then bridge $V_N^*(z^+)$ and $V_N^*(z)$ by bounding the difference between $V_N^*(z^+)$ and $V_N(\bar{z}^+, \ubb)$.
    
    We first bound the difference between $V_N(z^+, \ubb)$ and $V_N(\bar{z}^+, \ubb)$, which is also an upper bound for the difference between $V_N^*(z^+)$ and $V_N(\bar{z}^+, \ubb)$ from the principle of optimality. Given $z$ and $\mathbf{u}$, we can express the explicit form of $V_N(z, \ub)$~as 
    \[
    V_N(z, \ub) = (O_N z + T_N \ub)^\tr \bar{Q} (O_N z + T_N \ub) + \ub^\tr \bar{R}\ub
    \]
    where, following the propagation of the nominal Koopman linear embedding, we have 
    \[
    O_N = \begin{bmatrix}
        I & A & A^2 & \ldots & A^N
    \end{bmatrix}^\tr, \quad  
    T_N  =\begin{bmatrix}
0 & 0 & 0 & \cdots & 0 \\
B & 0 & 0 & \cdots & 0 \\
AB & B & 0 & \cdots & 0\\
\vdots & \vdots & \vdots & \ddots & \vdots \\
A^{N-1}B & A^{N-2}B & A^{N-3}B & \cdots & B \\ 
\end{bmatrix},
    \]
    and $\bar{Q} := \mathrm{diag}(\tilde{Q}, \tilde{Q}, \ldots, \tilde{Q}, \hP) \in \mathbb{R}^{(N+1)n_z\times(N+1)n_z}$, $\tilde{Q} = C^\tr Q C \in \mathbb{R}^{n_z\times n_z}$ and $\bar{R}:= \mathrm{diag}(R,$ $ R,\ldots, R) \!\in\! \mathbb{R}^{Nm \times Nm}$. 
    
    The difference between $V_N(z^+, \ubb)$ and $V_N(\bar{z}^+, \ubb)$ can be bounded~as
    \begin{equation}
        \label{eq:diff-all-state}
    \begin{aligned}
    & \ \|V_N(z^+, \ubb)-V_N(\bar{z}^+, \ubb)\|\\
    =&\ \|(O_N z^+ + T_N \bar{\ub})^\tr \bar{Q} (O_N z^+ + T_N \bar{\ub}) -(O_N \bar{z}^+ + T_N \bar{\ub})^\tr \bar{Q} (O_N \bar{z}^+ + T_N \bar{\ub})\| \\
    \le & \ 2\sigma_3 \|\bar{z}^+\|\|z^+\!-\!\bar{z}^+\| \! + \! \sigma_3 \|z^+ \! - \! \bar{z}^+\|^2 \! + \! 2 \sigma_4 \|\ubb\| \|z^+ \!-\!\bar{z}^+\| \\ 
    \le & \ 2\sigma_3 c_2 L \|x\|^2\!\! +\! \sigma_3 L^2 \|x\|^2\!\! +\! 2\sigma_4 \sqrt{c_1} L\|x\|^2 \! = \! c_4 \|x\|^2,
    \end{aligned}
    \end{equation}
    in which $\sigma_3, \sigma_4$ are maximum singular values of $O_N^\tr \bar{Q} O_N$ and $O_N^\tr \bar{Q} T_N$ and we have $c_4 = \sigma_3 L^2 + 2(\sigma_3 c_2  + \sigma_4 \sqrt{c_1})L$. Thus,~\cref{eq:diff-all-state} implies
    \begin{equation}
    \label{eqn:bound-nonlinear-cost}
    V_N^*(z^+) - V_N(\bar{z}^+, \ubb) \le  V_N(z^+, \ubb)-V_N(\bar{z}^+, \ubb)  \le  c_4 \|x\|^2.
    \end{equation}

    We then bridge $V_N^*(z^+)$ and $V_N^*(z)$ via $V_N(\bar{z}^+, \ubb)$. From \Cref{prop:rec-feasibility}, we have
    \begin{equation}
    \label{eqn:bound-linearized-cost}
     V_N(\bar{z}^+, \ubb) \le V^*_N(z) - \lambda_Q \|x\|^2.
    \end{equation}
    Combing~\cref{eqn:bound-nonlinear-cost} and~\cref{eqn:bound-linearized-cost}, we finally have
    \[
    V_N^*(z^+) - V_N^*(z) \le -\lambda_Q \|x\|^2 + c_4 \|x\|^2.
    \]
    Again, we treat $c_4$ as $c_4 = \alpha_2(L) := \sigma_3 L^2 + 2(\sigma_3 c_2  + \sigma_4 \sqrt{c_1})L \in \mathcal{K}$, thus, there exists $\delta_2 := \alpha
    _2^{-1}(\lambda_Q)$ and $c$ such that $c: = \lambda_Q - c_4 > 0$ for $L < \delta_2$ which implies
     \[
    V_N^*(z^+) - V_N^*(z) \le -c \|x\|^2, \ \forall x \in \mathcal{S}.
    \]

\end{document}